\documentclass[11pt,leqno,a4paper]{article}
\usepackage{amssymb,amsmath,amscd,graphicx,fontenc,amsthm,mathrsfs}
\usepackage[pdftex,bookmarks,colorlinks,breaklinks]{hyperref} 
\hypersetup{linkcolor=blue,citecolor=red}
\usepackage{multicol}
\usepackage[numbers]{natbib}

\newtheorem{theorem}{Theorem}[section]
\newtheorem{definition}[theorem]{Definition}
\newtheorem{lemma}[theorem]{Lemma}
\newtheorem{proposition}[theorem]{Proposition}
\newtheorem{corollary}[theorem]{Corollary}
\theoremstyle{definition}
\newtheorem{remark}[theorem]{Remark}

\newtheorem{question}{Question}
\newcommand{\FF}{\mathbb{F}}
\newcommand{\QQ}{\mathbb{Q}}
\newcommand{\RR}{\mathbb{R}}
\newcommand{\ZZ}{\mathbb{Z}}
\newcommand{\CC}{\mathbb{C}}

\newcommand{\G}{\mathcal{G}}
\newcommand{\zer}{\mathscr{Z}}
\newcommand{\reg}{\mathscr{R}}

\newcommand{\1}{\mathbf{1}}
\newcommand{\p}{\mathfrak{p}}
\newcommand{\C}{\mathcal{C}}
\newcommand{\qgraph}[1]{{\mathsf{QG} }_{#1}}
\newcommand{\hgraph}[1]{{\mathsf{HG} }(#1)}

\newcommand{\Q}{\mathcal{Q}}
\newcommand{\HH}{\mathcal{H}}

\newcommand{\bchi}{\chi_{_\mathrm{Bor}}}

\newcommand{\fword}[1]{{\it{#1}}}

\renewcommand{\d}[1]{\mathrm{d}#1}

\DeclareMathOperator{\Cay}{Cay}

\DeclareMathOperator{\GL}{GL}
\DeclareMathOperator{\M}{M}

\DeclareMathOperator{\md}{mod}
\usepackage[top=2.2cm,bottom=2.2cm,right=2cm,left=2cm]{geometry}
\title{On a generalization of the Hadwiger-Nelson problem}
\author{Mohammad Bardestani,   
Keivan Mallahi-Karai
}
\newcommand{\Addresses}{{% additional braces for segregating \footnotesize
  \bigskip
  \footnotesize

M.~Bardestani, \textsc{Department of Mathematics and Statistics, University of Ottawa, 585 King Edward, Ottawa, ON K1N 6N5, Canada.}\par\nopagebreak \textit{E-mail address:} { \tt mbardest@uottawa.ca} 
 
  \medskip

  K.~Mallahi-Karai, \textsc{Jacobs University Bremen, Campus Ring I, 28759 Bremen, Germany.}\par\nopagebreak \textit{E-mail address:} { \tt k.mallahikarai@jacobs-university.de} 
  }}
\begin{document}
\maketitle
%{\small\tableofcontents}
%===============================================================
\begin{abstract}
For a field $F$ and a quadratic form $Q$ defined on an $n$-dimensional vector space $V$ over $F$, let
$\qgraph{Q}$, called the quadratic graph associated to $Q$, be the graph with the vertex set $V$ where vertices $u,w \in V$ form an edge if and only if $Q(v-w)=1$. Quadratic graphs can be viewed as natural generalizations of the unit-distance graph featuring in the famous Hadwiger-Nelson problem. In the present paper, we will prove that for a local field $F$ of characteristic zero, the Borel chromatic number of $\qgraph{Q}$ is infinite if and only if $Q$ represents zero non-trivially over $F$. The proof employs a recent spectral bound for the Borel chromatic number of Cayley graphs, combined with an analysis of certain oscillatory integrals over local fields. As an application, we will also answer a variant of question 525 proposed in the 22nd British Combinatorics Conference 2009~\cite{Cam}.
\end{abstract}
\let\thefootnote\relax
\let\thefootnote\relax
%=====================================
\section{Introduction}

The celebrated Hadwiger-Nelson problem asks for the minimum number of colors required to color $\RR^n$ such that no two points at distance one from each other have the same color. Recall that the chromatic number of a graph $\G$, denoted by $\chi(\G)$, is the least cardinal $c$ such that the vertices of $\G$ can be partitioned into $c$ sets (called color classes) such that no color class contains an edge in $\G$. Hence, the Hadwiger-Nelson problem is the question of finding $\chi(\G_n)$, where $\G_n$ is the graph with vertex set $V(\G_n)=\RR^n$, where the adjacency of vertices $x,y \in V(\G_n)$ is defined by the equation $Q(x-y)=1$; here, $Q(x_1,\dots, x_n)=x_1^2+\dots+x_n^2$ is the canonical positive-definite quadratic form on $\RR^n$. Although it is not hard to show that 
\begin{equation}\label{HN}
4 \le \chi(\G_2) \le 7,
\end{equation} the question of determining the exact value of $\chi(\G_2)$ remains an open problem to date. For a discussion of historical developments on this problem and its extension to $\chi(\G_n)$ we refer the reader to~\cite{Soifer,Sz}. 

The definition of chromatic number allows the color classes to be arbitrary subsets of the vertex set. When the vertex set is equipped with an additional structure (e.g., a topology, or a $\sigma$-algebra structure), one can define a variant of the chromatic number by imposing further constraints on the color classes related to this structure. The {\it Borel chromatic number} of a topological graph $\G$, denoted by $ \bchi( \G)$, is the least cardinal $c$ such that the set $V(\G)$ can be partitioned into $c$ Borel subsets none of which contains two adjacent vertices of $\G$. Once the Borel measurability assumption is imposed on the color classes, a wide range of analytical tools will become available, in order to establish bounds on the Borel chromatic number of the graphs in question. One of the pioneering papers in this direction is the work of Falconer~\cite{Falconer} who proved that under the measurability assumption one can improve \eqref{HN} to give $\bchi(\G_2) \ge 5$. These techniques have been further developed by Bukh~\cite{Bukh}, Bachoc--Nebe--Oliveira--Vallentin~\cite{BNOV} and  Oliveira--Vallentin~\cite{Vallentin}. 

The main question addressed in this paper concerns a new generalization of the Hadwiger-Nelson problem. Let $F$ be an arbitrary field of characteristic different from $2$, and let $V$ be a finite-dimensional vector space over $F$. Given a non-degenerate quadratic form $Q$ defined on $V$, we will use the term {\it quadratic space} to refer to the pair $\Q=(V,Q)$. 
For a quadratic space $\Q=(V,Q)$, we can now define the {\it quadratic graph} $\qgraph{\Q}$ as follows. The vertex set of 
$\qgraph{\Q}$ coincides with the vector space $V$. Vectors $v,w \in V$ are adjacent in $\qgraph{\Q}$ if and only if $Q(v-w)=1$. More succinctly, $\qgraph{\Q}$ is the Cayley graph of the abelian group $V$ with respect to the unit sphere $S=\{v\in V: Q(v)=1\}$. 
Recall that a quadratic form $Q$ defined over $V$ is {\it isotropic} if there exists a non-zero vector $v \in V$ with $Q(v)=0$. Otherwise, $Q$ is called {\it anisotropic}. The main result of this paper is the following:

\begin{theorem}\label{Borel-main-theorem} Let $F$ be a local field of characteristic zero, $V$ a finite-dimensional $F$-vector space of dimension at least $2$, and $\Q=(V,Q)$ the quadratic space associated to a non-degenerate quadratic form $Q$ on $V$. Then $\bchi(\qgraph{\Q})$ is finite if and only if $Q$ is anisotropic.
\end{theorem}

A local field is, by definition, a topological field with a non-discrete locally compact topology. It is well-known that a local field of characteristic zero is isomorphic to the field of real or complex numbers, or to a finite extension of the field of $p$-adic numbers $\QQ_p$, where $p$ varies over the set of prime numbers. Theorem \ref{Borel-main-theorem} shows in particular that 
whether or not the form $Q$ is isotropic can be read off the Borel chromatic number of the associated quadratic graph. One can likewise inquire about the chromatic number of quadratic graphs over number fields. For the quadratic space $\Q=(\mathbb{Q}^2, x^2+y^2)$, Woodall~\cite{Woodall} has shown that $\chi(\qgraph{\Q})=2$. Using the local-global principle, we can deduce a purely algebraic corollary of Theorem~\ref{Borel-main-theorem} that applies to the {\it ordinary} chromatic number of quadratic forms over number fields. 
\begin{corollary}
Let $K$ be a number field, $V$ a finite-dimensional $K$-vector space of dimension at least $2$, and $\Q_K=(V,Q)$ the quadratic space associated to a non-degenerate quadratic form $Q$ on $V$. If $Q$ is anisotropic, then $\chi(\qgraph{\Q_K})$ is finite.
\end{corollary}

\begin{proof}
By Hasse-Minkowski theorem, there exists a place $\nu$ of $K$ such that the form $Q(x)$ is anisotropic over the completion $F=K_{\nu}$. 
Since $F$ is a local field, by Theorem~\ref{Borel-main-theorem}, we have $\bchi(\qgraph{\Q_F})< \infty$, where $\Q_F$ is the 
quadratic space on $K$ defined by $Q$. Hence, \fword{a fortiori}, 
$\chi(\qgraph{\Q_K}) \le \chi(\qgraph{\Q_F}) \le \bchi(\qgraph{\Q_F})< \infty$.
\end{proof}

Let us describe the key ideas of the proof of Theorem~\ref{Borel-main-theorem}. For anisotropic forms, a partitioning of the space $F^n$ into box-like subsets, and periodically coloring these boxes yields a proper coloring. The main technical part of the work goes into the proof of the theorem for isotropic forms. Here we use a spectral bound established in~\cite{SBO} which provides a lower bound on the Borel chromatic number of Cayley graph of vector space $V$ with respect to Borel measurable set $S=\{v\in V: Q(v)=1\}$, which is analogous to the well-known Hoffman bound for finite graphs. This bound, \fword{per se}, is only applicable when $S$ is bounded, which is never the case here. Instead, we will approximate $S$ by an increasing family $\{ S_T \}_{T \ge 1}$ of bounded Borel sets. The crucial step in the proof will consist of
establishing a useful lower bounds for the a family of oscillatory real and $p$-adic integrals that arise as the Fourier transform of carefully chosen probability measures on $\{ S_T \}$. The latter is dealt with using van der Corput's lemma in the real case and certain Bessel functions in the $p$-adic case. 

In certain cases, one can attempt to quantify Theorem~\ref{Borel-main-theorem} for anisotropic forms. Let $p>2$ be a prime number and set $\Q_p=(\QQ_p^2, x_1^2+x_2^2)$. Note that 
if $p \equiv 3 \pmod{4}$, then the quadratic form $x_1^2+x_2^2$ is anisotropic over $\QQ_p$ and hence $\bchi(\Q_p)< \infty$. Indeed, by keeping track of the bounds in the proof of Theorem~\ref{Borel-main-theorem} for anisotropic quadratic forms, one can see that $\bchi(\qgraph{\Q_p}) =O(p^2)$. We can now pose the following question:

\begin{question}[The $p$--adic Hadwiger-Nelson problem] Is it true that 
there exists a universal constant $B$ (independent of $p$) such that $$\bchi(\qgraph{\Q_p})\leq B,$$ holds as $p$ ranges over all primes of the form $4k+3$?
\end{question}

\begin{remark}[Some set theoretic considerations]
Let $S$ be the subset of $\RR^2$ defined by 
\[ S= \{ (q_1,q_2) \pm \epsilon: q_1,q_2 \in \QQ,\, \epsilon= ( \sqrt{2}, 0), (0, \sqrt{2}), 
( \sqrt{2}, \sqrt{2}), (- \sqrt{2}, \sqrt{2}) \}. \]
Soifer and Shelah~\cite[section 3]{Shelah} define the graph $G_2$ as the Cayley graph of $\RR^2$ with respect to $S$, and prove the intriguing theorem that the value of the chromatic number of $G_2$ depends on the choice of axioms for set theory. More precisely, they showed that under {\bf ZFC}, $ \chi(G_2)=4$ holds, while if one opts for the axiomatic system (the consistency of this system is proved by Solovay~\cite{Solovay}) {\bf ZF}+ {\bf DC}+ {\bf LM} (that is, {\bf ZF} combined with the axiom of dependent choice, and the axiom that every subset of $\RR$ is Lebesgue measurable), then $G_2$ does not even admit a proper coloring with $\aleph_0$ colors. Let $\G_{n,r}$ denote the quadratic graph associated the canonical quadratic form of signature $(r,n-r)$ over $\RR^{n}$. Theorem \ref{Borel-main-theorem} shows that if $1 \le r \le n-1$, then
\begin{equation}\label{ineq}
\chi(\G_{n,r}) \le \bchi(\G_{n,r})= \aleph_0.
\end{equation}

If the value of $ \chi( \G_{n,r})$ turns out to be finite, then one obtains another (in our opinion more straight-forward, though less dramatic) example of a topological graph with the type of behavior \`a la Soifer-Shelah. 
\end{remark}

For a field $F$ and an integer $n \ge 2$, {\it the regular graph} associated to the  ring of $n\times n$ matrices over $F$, denoted by $\Gamma_n(F)$, has the vertex set $\GL_n(F)$ (general linear group over $F$), where two vertices $x,y \in \GL_n(F)$ are adjacent if $\det(x+y)=0$. Regular graphs (for arbitrary rings) were introduced in~\cite{anderson} and the study of their chromatic number was initiated in~\cite{AJF}. For motivations and more details, we refer the reader to Section \ref{application}.
As a corollary of Theorem~\ref{Borel-main-theorem}, we answer question 525 from~\cite{Cam} by proving the following theorem:

\begin{theorem}\label{app}
Let $F$ be a local field of characteristic zero and $n \ge 2$. Then the Borel chromatic number of the regular graph $\Gamma_{n}(F)$ is infinite. 
\end{theorem}

\begin{remark} It is very likely that Theorems~\ref{Borel-main-theorem} and \ref{app} equally hold over local fields of positive characteristic. The details need to be worked out. 
\end{remark}

This article is organized as follows. In section \ref{prel} we will set some notation and briefly explain the ideas behind the spectral bound used later. Sections \ref{sec-quadratic} and \ref{application} are devoted to the proofs of Theorem \ref{Borel-main-theorem} and Theorem \ref{app}.

%=====================================================================
\section{Preliminaries}\label{prel}
In this section, we will review some of the ideas in the beautiful recent work~\cite{SBO} to prove an analog of the Hoffman bound for certain Cayley graphs of the additive group of the Euclidean space $\RR^n$. 
\subsection{The chromatic number of self-adjoint operators}
Let $\G=(V,E)$ be a finite graph and consider the Hilbert space $L^2(V)$ of all complex-valued functions on $V$, equipped with the inner product 
$\langle f,g\rangle:=\sum_{v\in V}f(v)\overline{g(v)}$.
The adjacency operator $A: L^2(V)\to L^2(V)$ defined by $Af(v)=\sum_{vw\in E} f(w)$ is self-adjoint. It is easy to see that the {\it numerical range} of $A$  
$$
W(A):=\{\langle Af,f\rangle: \|f\|_2=1\},
$$  
is equal to $[\lambda_{n-1},\lambda_0]$, where $\lambda_0\geq\lambda_1\geq\cdots\geq\lambda_{n-1}$ is the spectrum of $A$. It is also easy to verify that a subset $I\subseteq V$ is independent if and only if for all $f\in L^2(V)$ supported on $I$, one has $\langle Af,f\rangle=0$.
This novel interpretation of independent sets is used in~\cite{SBO} to prove an analog of the Hoffman bound for certain Cayley graphs of the additive group of the Euclidean space $\RR^n$. As we will need to work in a slightly more general framework, it will be useful to briefly go over some of the key points of~\cite{SBO}; for details, we refer the reader to the original paper.  
     
Let $(V, \Sigma, \mu)$ be a finite measure space, consisting of a set $V$, a $\sigma$-algebra $\Sigma$ on $V$, and a measure $\mu$. Consider the Hilbert space $L^2(V)$ with the inner product:
\begin{equation*}
L^2(V)=\{f: V\to \CC\,\, \text{measurable}: \int_V |f|^2\, \d{\mu}<\infty\},\qquad \langle f,g\rangle=\int_V f(x)\overline{g(x)}\,\d{\mu(x)}.
\end{equation*}
For a bounded and self-adjoint operator $A: L^2(V)\to L^2(V)$, one can show that the {\it numerical range} of $A$ defined by 
$$
W(A)=\{\langle Af,f\rangle: \|f\|_2=1\},
$$
is an interval in $\RR$. We denote the endpoints of $W(A)$ by 
$$m(A):=\inf\{\langle Af, f\rangle: \|f\|_2 = 1 \}, \qquad M(A):=\sup\{\langle Af, f\rangle: \|f\|_2=1\}.$$

\begin{definition} Let $A: L^2(V)\to L^2(V)$ be a bounded, self-adjoint operator. A
measurable set $I\subseteq V$ is called an {\it independent set for $A$} if $\langle Af, f\rangle=0$ for each $f\in L^2(V)$ which vanishes almost everywhere outside of $I$. Moreover  the {\it chromatic number} of $A$, denoted by $\chi(A)$, equals the least number $k$ such that one can partition $V$ into $k$ independent sets of $A$.
\end{definition}
The following result~\cite[Theorem 2.3]{SBO} is analogous to the well-known Hoffman bound for finite graphs: 
\begin{theorem}\label{chrom-op}
Let $A: L^2(V)\to L^2(V)$ be a nonzero, bounded and self-adjoint operator. Moreover assume that $\chi(A)<\infty$. Then
\begin{equation}\label{chrom-A}
\chi(A)\geq 1-\frac{M(A)}{m(A)}.
\end{equation}
\end{theorem}
\begin{remark}\label{positive-negative}
Similar to Hoffman bound, when $A\neq 0$ and $\chi(A)<\infty$ the proof of the above theorem shows that $m(A)<0$ and $M(A)>0$.
\end{remark}
We now relate this theorem to the Borel chromatic number of Cayley graphs with the vertex set $F^n$, where $F$ is $\mathbb{R}$ or a $p$-adic field $\mathbb{Q}_p$. All that follows is parallel to~\cite{SBO}, where the case $F=\RR$ is dealt with; in fact the arguments in~\cite{SBO} carry over to the $p$-adic case without any effort. For the convenience of the reader we will briefly discuss the key points.
For the rest of this subsection, let $F$ be either $\RR$ or a $p$-adic field $\QQ_p$ with the ring of integers $\mathbb{Z}_p$. We will also assume that $\mathcal{S}\subseteq F^n$ is a symmetric (i.e., $-\mathcal{S}=\mathcal{S}$), Borel measurable set (with respect to the $\sigma$-algebra generated by the open sets), which does not contain the origin in its closure. The Cayley graph $\Cay(F^n,\mathcal{S})$ has the vertex set $F^n$, where two vertices $x,y\in F^n$ are adjacent if and only if $x-y\in\mathcal{S}$. 
 The goal is to find a spectral bound for the Borel chromatic number of this graph. 

Let $\mu$ be a finite Borel measure on $F^n$ with support contained in $\mathcal{S}$. Therefore $\mu$ is a Radon measure~\cite[Theorem 7.8]{Folland}. We will also assume that $\mu$ is symmetric, i.e., $\mu(-\mathscr{A}) =\mu(\mathscr{A})$ holds for all Borel measurable sets $\mathscr{A}$. This measure produces a bounded, self-adjoint operator 
$$
A_\mu: L^2(F^n)\to L^2(F^n),\qquad f\mapsto f*\mu,$$
where  
$f*\mu(x)=\int_{F^n} f(x-y)\,d\mu(y)$.
One can easily see that since the support of $\mu$ is contained in $\mathcal{S}$, any Borel measurable independent set $I$ of $\Cay(F^n,\mathcal{S})$  is also an independent set of $A_\mu$ (see~\cite[Section 3.1]{SBO}). One can show that the Borel chromatic number of $\Cay(F^n,\mathcal{S})$ is finite when $\mathcal{S}\subseteq F^n$ is a bounded set with the above properties. This will be proved later (see Lemma~\ref{finite-chro}). Therefore by Theorem~\ref{chrom-op} we have
\begin{equation}\label{Eq4}
1-\frac{M(A_\mu)}{m(A_\mu)}\leq \chi(A_\mu)\leq \bchi(\Cay(F^n,\mathcal{S})).
\end{equation}
The numerical range of $A_\mu$ can now be determined using Fourier analysis. Below, we will review some of the basic properties of the Fourier transform over local fields. For more detail we refer the reader to~\cite{Rudin,Taibleson}.
For $F=\RR$, we will use the character  
$$\psi: \RR \to \CC^{\ast}, \quad \psi(x)=\exp(2\pi i x).$$ When $F=\QQ_p$, for $x\in\QQ_p$, we write $n_x$ for the smallest non-negative integer such that $p^{n_x}x\in\ZZ_p$. Let $r_x\in\ZZ$ be such that $r_x\equiv p^{n_x}x \pmod{p^{n_x}}$. It is well-known that the following map (called the Tate character)
\begin{equation}\label{additive-char}
\psi: \QQ_p\to \CC^*,\qquad \psi(x)=\exp \left( {\frac{2\pi i r_x}{p^{n_x}} } \right),
\end{equation}
is a non-trivial character of $(\QQ_p,+)$ with the kernel $\ZZ_p$. 
We will denote by $\d{x}$ a Haar measure on $F$. When $F=\QQ_p$, we assume that it is normalized such that $\int_{\ZZ_p} \d{x}=1$. We will also use $\d{x}$ for the product (Haar) measure $\d{x_1} \cdots \d{x_n}$ on $F^n$. Moreover for a Borel set $E\subseteq F^n$ we use $|E|=\int_E \d{x}$ to denote the measure of $E$.  

The Fourier transforms of $f \in L^1(F^n)$ and $\mu$ are respectively defined by the integrals
$$
\widehat{f}(u)=\int_{F^n}\overline{\psi}(x\cdot u)f(x)\, \d{x},\qquad \widehat{\mu}(u)=\int_{F^n} \overline{\psi}( x\cdot u)\, \d{\mu(x)}.
$$
Here $x\cdot u$ is the standard bilinear form on $F^n$, and $ \overline{\psi }$ is the complex conjugate of $\psi$. 
By Plancherel's theorem the Fourier transform extends to an isometry on $L^2(F^n)$ and so for any $f\in L^2(F^n)$ we have
\begin{equation}\label{mult-op}
\langle A_\mu f,f \rangle=\langle \widehat{A_\mu f},\widehat{f} \rangle=\langle \widehat{f*\mu},\widehat{f} \rangle=\langle \widehat{\mu}\widehat{f},\widehat{f}\rangle. 
\end{equation}

\begin{lemma}\label{m=inf}
Let $\mu$ be a symmetric finite Borel measure on $F^n$ with support contained in $\mathcal{S}$. Then the numerical range of $A_\mu$ is given by
$$
m(A_\mu)=\inf_{u\in F^n}\widehat{\mu}(u),\qquad M(A_\mu)=\sup_{u\in F^n}\widehat{\mu}(u).
$$
\end{lemma}

\begin{proof}
From \eqref{mult-op} combined with the fact that the Fourier transform on $L^2(F^n)$ is an isometry, we   
deduce that the numerical range of the operator $A_\mu$ is the same as the numerical range of the multiplicative operator $g\mapsto \widehat{\mu} g$. Since $\mu$ is a symmetric finite Borel measure, then $\widehat{\mu}$ is a continuous, real-valued and bounded function. Now let $g\in L^2(F^n)$ with $\|g\|_2=1$. Evidently 
$$
\langle \widehat{\mu} g,g\rangle=\int_{F^n} \widehat{\mu}(x)|g(x)|^2\, \d{x}\geq \inf_{u\in F^n}\widehat{\mu}(u),
$$
and so $m(A_\mu)\geq \inf_{u\in F^n}\widehat{\mu}(u).$
Now let $\varepsilon>0$ and pick $x_0\in F^n$ with $\widehat{\mu}(x_0)\le\inf_{u\in F^n}\widehat{\mu}(u)+\varepsilon$. Let $B_{ \delta}$ be a ball of radius $ \delta$ centered at $x_0$. Since $\widehat{\mu}$ is continuous then we have 
$$
\lim_{ \delta \to 0}\int_{F^n}\widehat{\mu}(x)\left|
\frac{\1_{B_{ \delta}}(x)}{\sqrt{|B_{ \delta}|}}\right|^2\, \d{x}=\lim_{ \delta \to 0} \frac{1}{|B_{ \delta}|} \int_{B_{ \delta}}\widehat{\mu}(x)\, \d{x}=\widehat{\mu}(x_0),
$$  
where $\1_{B_{ \delta}}$ is the characteristic function of $B_{ \delta}$. Hence 
$$m(A_\mu)\leq \widehat{\mu}(x_0)\leq \inf_{u\in F^n}\widehat{\mu}(u)+\varepsilon,$$ and therefore we have $m(A_\mu)=\inf_{u\in F^n}\widehat{\mu}(u)$. Similarly we have $M(A_\mu)=\sup_{u\in F^n}\widehat{\mu}(u)$.
\end{proof}
We can now summarize these in the following theorem.
\begin{theorem}[Bachoc-DeCortede-Oliveira-Vallentin~\cite{SBO}]\label{spectral}
Let F be either $\RR$ or a p-adic field $\QQ_p$ and let $\mathcal{S}$ be a bounded, symmetric (i.e., $-\mathcal{S}=\mathcal{S}$) Borel measurable subset of $F^n$ which does not contain the origin in its closure. Then for any symmetric, finite Borel measure $\mu$ on $F^n$ with support contained in $\mathcal{S}$ we have
\begin{equation}
1-\frac{\sup_{u\in F^n}\widehat{\mu}(u)}{\inf_{u\in F^n}\widehat{\mu}(u)}\leq \bchi(\Cay(F^n,\mathcal{S})).
\end{equation}
\end{theorem}
%==================================================
\subsection{Quadratic and hyperbolic graphs}\label{ssn:hyp-graphs}
We first recall the general notion of a quadratic space. For more details we refer the reader to~\cite[Chapter IV]{Serre-Arith}. Let $F$ be an arbitrary field of characteristic different from $2$ and let $V$ be a finite dimensional vector space over $F$.  A function $Q: V\to F$  is  called a quadratic form  on  $V$  if
\begin{itemize}
\item[(a)] $Q(ax)=a^2 Q(x)$ for $a\in F$  and $x\in V$.
\item[(b)] The map $(x, y)\mapsto Q(x+y)-Q(x)-Q(y)$ is a  bilinear form.
\end{itemize}
Such a  pair $\Q=(V,Q)$ is called  a quadratic space. For each quadratic space $\Q=(V,Q)$ we then obtain a symmetric bilinear form on $V$ defined by
\begin{equation}\label{bili}
\langle x,y\rangle:=\frac{1}{2}\left(Q(x+y)-Q(x)-Q(y)\right)=\frac{1}{2}(Q(x)+Q(y)-Q(x-y)).
\end{equation}
We say $\Q=(V,Q)$ is non-degenerate when the bilinear form~\eqref{bili} is non-degenerate. In this paper, we will only
consider non-degenerate quadratic spaces.  
If $\Q=(V, Q)$ and $\Q'=(V', Q')$ are two quadratic spaces, a  linear map 
$f:V\to   V'$ such  that  $Q'\circ f=Q$  is  called  a  morphism  (or metric morphism)  of $\Q$ into $\Q'$; we say $\Q$ is isomorphic to $\Q'$ when $f$ is a vector space isomorphism.
\begin{definition}
For a given a quadratic space $\Q=(V,Q)$, the  quadratic graph, denoted by $\qgraph{\Q}$, is defined as follows. The vertex set of 
$\qgraph{\Q}$ coincides with the vector space $V$. Vectors $v,w \in V$ form an edge in $\qgraph{\Q}$ if and only if $Q(v-w)=1$.
\end{definition}

Let $\Q=(V,Q)$ be a non-degenerate quadratic space over $F$. It is well known that  $Q$ is isotropic if and only if it can be expressed in an appropriate basis $(e_i)_{1 \le i \le n}$ for $V$ as 
\begin{equation}\label{eq:canonical}
 Q(x_1e_1+\dots+x_ne_n)= x_1x_2+a_3 x_3^2+\dots+a_nx_n^2. 
\end{equation}
In particular, the restriction
of $Q$ to the $2$-dimensional subspace $H$ spanned by $e_1$ and $e_2$ takes the form 
\begin{equation}\label{15}
Q(x_1e_1+x_2e_2)=x_1x_2.
\end{equation}
 The (unique up to isomorphism) quadratic space $\HH=(H, Q|_{H})$ (or, $\HH(F)$ if the underlying field $F$ is to be emphasized) is often referred to as the {\it hyperbolic plane} (over $F$). The associated quadratic graph to the hyperbolic plane (over the field $F$) will be called the {\it hyperbola graph over $F$}, and denoted by 
$\hgraph{F}$. In other words,
$$
\hgraph{F}=\Cay(F^2,\mathcal{S}),\qquad \mathcal{S}=\{(x,y)\in F^2: xy=1\}.
$$
Note that the map $F^2\to V$ defined by $(x_1,x_2)\mapsto x_1e_1+x_2e_2$ embeds $\hgraph{F}$ as an induced subgraph of $\qgraph{\Q}$ when $Q$ is an isotropic quadratic written in the form~\eqref{15}. If $F$ is moreover a topological field, the image of this map is a closed subset of $V$. Therefore we have the following simple fact. 
\begin{lemma}\label{hyper-quad} Let $F$ be an arbitrary field of characteristic different from $2$ and let $\Q=(V,Q)$ be an isotropic quadratic space over $F$. Then $\chi(\qgraph{\Q})\geq \chi(\hgraph{F})$.  Moreover, if $F$ is a topological field, then
$$\bchi(\qgraph{\Q}) \ge \bchi(\hgraph{F}).$$
\end{lemma}
%=================================================
\section{Proof of the main theorem}\label{sec-quadratic} In this section we will prove Theorem~\ref{Borel-main-theorem}. 
In the first two subsections, we will give the proof for the real and $p$-adic isotropic forms. In the last subsection, we will deal with the (easier) case of anisotropic forms. Before embarking on the proof in the isotropic case, we would like to show that Theorem~\ref{Borel-main-theorem} cannot be proven by simply pointing to an infinite clique. Recall that the clique number of a graph $\G$, denoted by $\omega(\G)$ is the largest cardinal $c$ such that $\G$ contains a complete subgraph on $c$ vertices. We now have

\begin{proposition}\label{quadratic-clique}
Let $\Q=(V,Q)$ be a quadratic space over an arbitrary field $F$ of characteristic different from $2$. Then the clique number of $\qgraph{\Q}$ is bounded above independently of the ground field. More precisely, 
$$\omega(\qgraph{Q}) \le \dim V +1.$$
\end{proposition}

\begin{proof}
Suppose $v'_0, v'_1, \dots, v'_n \in V$ form a clique in $\qgraph{Q}$. For $1\leq i\leq n$, set $v_i=v'_i-v'_0$. Hence $Q(v_i)=1$, $1 \le i \le n$ and $Q(v_i-v_j)=1$ for $1 \le i\neq j \le n$. From~\eqref{bili} we deduce that $\langle v_i,v_j\rangle=\frac{1}{2}$ for $i \neq j$. For $1\leq i\leq n$, consider the following functionals 
$$
\alpha_i: V\to F, \qquad w\mapsto \langle w,v_i\rangle.
$$
 We claim that $ \alpha_1, \dots, \alpha_n$ are linearly independent. In fact, evaluating $ \sum_{k=1}^{n} c_k \alpha_k=0$ at $w=v_i$ yields 
\[\frac{1}{2}(c_1+ \cdots+ c_n)+\frac{1}{2}c_i=0, \quad 1 \le i \le n. \] 
This system of linear equations has the unique solution $c_1= \cdots = c_n=0$. From here, we conclude that $n \le \dim V^*=\dim V$, whence 
$\omega(\qgraph{Q}) \le \dim V +1$. 
\end{proof}

\begin{remark}
The $n+1$ vertices of a regular simplex of side $1$ in $\RR^n$ form a clique in the quadratic graph of the  form $x_1^2+ \cdots + x_n^2$; this shows that the above bound is indeed sharp. 
\end{remark}

Let us recall a number of basic facts about local fields. A local field is, by definition, a non-discrete locally compact topological field. It is a well-known fact that an archimedean local field is isomorphic to the field of real or complex numbers. A non-archimedean local field is a complete field with respect to a discrete valuation that has a finite residue field, and is isomorphic either to a finite extension of $\QQ_p$ ($p$ is a prime number) or to the field of formal Laurent series $\FF_q((t))$ over a finite field with $q=p^f$ elements. A non-archimedean local field $F$ comes with its ring of integers $\mathcal{O}$, which is a local ring whose unique prime ideal $\mathfrak{p}$ is principal. Let $\varpi$ be a generator for $\mathfrak{p}$, often referred to as a uniformizer.  
For $m\in\ZZ$, we write $\mathfrak{p}^m$ for the ideal generated by $\varpi^m$. We have thus the following filtration of $F$:
$$
F\supseteq\cdots\supseteq\mathfrak{p}^{-2}\supseteq \mathfrak{p}^{-1}\supseteq \mathfrak{p}^0=\mathcal{O}\supseteq\mathfrak{p}\supseteq\mathfrak{p}^2\supseteq\cdots \supseteq \{0\}.
$$

Let $\| \cdot \|$ denote the standard norm on $F$. When $F$ is the field of real or complex numbers, this 
denotes the standard Euclidean norm. In the non-archimedean case, this norm is defined as follows: 
Every element of $x \in F$ has a unique representation as $x=u\varpi^m$, where $u$ is a unit of $\mathcal{O}$ and $m \in \ZZ$. Let $q$ denote the cardinality of the residue field $\mathcal{O}/\mathfrak{p}$. The norm function $\| \cdot \|$ is defined by 
$\| x \|=q^{-m}$. 

\begin{lemma}\label{finite-chro} Let $F$ be a local field and let $\mathcal{S}\subseteq F^n$ be a bounded, symmetric and Borel measurable set which does not contain the origin in its closure. Then $\bchi(\Cay(F^n, \mathcal{S}))< \infty$.
\end{lemma} 
\begin{proof} First, we equip $F^n$ with the maximum norm define by
\[ \|{\bf x}\|:= \max( \|x_1\|, \dots , \|x_n\|), \qquad {\bf x}=(x_1, \dots, x_n). \]
For $0<c_1<c_2$, define the annulus with radii $r_1$ and $r_2$ by $S_{c_1,c_2}= \{ {\bf x} \in F^n: c_1 < \|{\bf x} \| < c_2 \}$. 
Since $\mathcal{S}$ is bounded and $0\notin \overline{\mathcal{S}}$, it follows that $ \mathcal{S} \subseteq S_{c_1, c_2}$ for some values of $c_1, c_2 >0$. As $\Cay(F^n,\mathcal{S})$ is isomorphic to $\Cay(F^n, (1/c_1) \mathcal{S})$, up to passing to a dilation of $\mathcal{S}$, we may assume that $c_1=1$. 
Let us first consider the case $F=\RR$. Fix an integer $m>c_2+2$. We claim that the map
\[ c({\bf x})=( \left\lfloor x_1 \right\rfloor \md {m}, \dots, \left\lfloor x_n \right\rfloor \md m),\qquad {\bf x}=(x_1, \dots, x_n).\]
provides a proper coloring. Assume $c({\bf x})=c({\bf y})$;
there are two possibilities: either $\left\lfloor x_i \right\rfloor = 
\left\lfloor y_i \right\rfloor$ for all $1 \le i \le n$, which implies $\|{\bf x}-{\bf y}\|<1$, showing that ${\bf x}$ and ${\bf y}$ are not adjacent. Or, for some $1 \le i \le n$, we have $| \left\lfloor x_i \right\rfloor - \left\lfloor y_i \right\rfloor| \ge m$, which, in turn, implies that $\|{\bf x}- {\bf y}\| \ge |x_i-y_i|>m-2>c_2$, proving that ${\bf x}$ and ${\bf y}$ are not adjacent. This gives the upper bound $\bchi(\Cay(\RR^n,\mathcal{S}))\leq m^n$.

Now let $F$ be a non-archimedean local field with the residue field $\mathcal{O}/\mathfrak{p}$ of size $q$. Let $\varpi$ be a uniformizer, so  $\| \varpi \|=1/q$.  Set $A=\mathcal{O}\times\dots\times\mathcal{O}\subseteq F^n$. Since $c_1=1$, we have $\mathcal{S} \cap A =\emptyset$.  Choose an integer $m$ such that $q^m\geq c_2$. Each $x \in F$ has a unique $\varpi$-adic representation~\cite[4.4 Proposition]{Neukirch} of the form
$$x= \sum_{j=r }^{ \infty} b_j \varpi^j, \qquad r\in \ZZ,\; b_j\in\FF_q.$$
Define
\[ \{x\}_m= \sum_{-m\leq j\leq 0}b_j\varpi^j, \]
and consider the map 
\[ c({\bf x})= ( \{x_1\}_m , \dots,  \{x_n\}_m),\qquad {\bf x}=(x_1, \dots, x_n). \]
Suppose $c({\bf x})=c({\bf y})$. This implies that in the expansion of each entry of ${\bf x}-{\bf y}$ there are no terms $\varpi^j$ with $-m \le j \le 0$. If for some $ 1 \le i \le n$, $x_i-y_i$ contains $\varpi^r$ for some $r<-m$, then $\|{\bf x}-{\bf y} \| \ge q^m>c_2$. Otherwise, $x_i-y_i$, for all $1\leq i\leq n$, only contain $\varpi^{j}$ for $j>0$, hence $\|{\bf x}- {\bf y}\|<1$. In either case, we have ${\bf x}-{\bf y} \not\in \mathcal{S}$. Therefore $\bchi(\Cay(F^n,\mathcal{S}))\leq q^{(m+1)n}$.
\end{proof}
%==========================================================
\subsection{The case of isotropic real quadratic forms} 
The aim of this subsection is to show that if $\Q=(V,Q)$ is an isotropic quadratic space over $\RR$, then the Borel chromatic number of $\qgraph{\Q}$ is infinite. From Lemma~\ref{hyper-quad}, it suffices to show that the Borel chromatic number of $\hgraph{\RR}$ is infinite. Fix a parameter $T > 0$ and consider the measurable set 
\begin{equation}
\mathcal{S}_T=\left\{(x,y)\in \RR^2: xy=1,\, e^{-T}\leq |x|\leq e^T\right\}.
\end{equation}
We will refer to the Cayley graph $\Cay(\RR^2,\mathcal{S}_T)$ as the {\it T-chopped hyperbola graph}. 
\begin{theorem}\label{spec-real} For $T >0$, let $\Cay(\RR^2,\mathcal{S}_T)$ be the T-chopped hyperbola graph. Then 
\begin{equation}
\bchi(\Cay(\RR^2,\mathcal{S}_T))\geq \frac{T}{8\sqrt{2/\pi}}+1.
\end{equation}
\end{theorem}

Notice that $\mathcal{S}_T$ is a bounded symmetric set, i.e., $\mathcal{S}_T=-\mathcal{S}_T$  which does not contain the origin in its topological closure. Hence the Borel chromatic number of $\Cay(\RR^2,\mathcal{S}_T)$ is finite by Lemma~\ref{finite-chro}. 
For a measurable set $E\subseteq \RR^2$, we define the following measure 
$$
\mu_T(E)=\frac{1}{4T}\left(\int_{e^{-T}}^{e^T} \frac{\1_E(s,1/s)}{s}\, \d{s}+\int_{e^{-T}}^{e^{T}} \frac{\1_{(-E)}(s,1/s)}{s}\, \d{s}\right),
$$  
where $\1_E$ is the characteristic function of $E$. It is easy to verify that $\mu_T$ is a symmetric probability measure  supported on $\mathcal{S}_T$.  
\begin{lemma}\label{ocslliatory} For $T>0$ and arbitrary $(x,y)\in \RR^2$ we have
\[  - \frac{8 \sqrt{2/\pi}}{T} \le \widehat{\mu_T}(x,y)\]
\end{lemma}
For the proof we will need the van der Corput's lemma~\cite[Chapter VIII, Proposition 2]{Stein}.
\begin{lemma}\label{vdC}
Let $ \phi: [a,b] \to \RR$ be a function satisfying $|\phi''(t)|\ge \lambda>0$ for all $t \in [a,b]$. Then 
\[ \left|   \int_{a}^{b} e^{i \phi(t)}\, \d{t}  \right| \le 8 \lambda^{-1/2}. \] 
\end{lemma}
\begin{proof}[Proof of Lemma~\ref{ocslliatory}]
A simple calculation shows that 
$$
\widehat{\mu_T}(x,y)=\frac{1}{4T}\left(\int_{e^{-T}}^{e^T} \frac{e^{-2\pi i (sx+y/s)}}{s}\, \d{s}+\int_{e^{-T}}^{e^T} \frac{e^{2\pi i (sx+y/s)}}{s}\, \d{s}\right)=\frac{1}{2 T}\int_{e^{-T}}^{e^T} \frac{\cos (2\pi(xs+y/s))}{s}\, \d{s}.
$$
Using the substitution $s=e^{t}$, we have
 \[ \widehat{\mu_T}(x,y)=\frac{1}{2 T}\int_{- T}^{  T} \cos( 2\pi(xe^{t}+ y e^{-t})) \, \d{t}. \] 
Given $x,y \in \RR$, define 
\[ I= \left\{ -T \le t \le T: - \frac{1}{2} \le 2(xe^{t}+ y e^{-t})\le \frac{1}{2} \right\}, \]
and $I'= [-T, T ] \setminus I$.
Clealy $I$ is a closed interval and $I'$ is the union of at most two 
intervals $I'_{1}$ and possibly $I'_{2}$. Let $\phi(t)=2\pi(xe^{t}+ y e^{-t})$. By definition, for $t \in I'$, we have
$ |\phi''(t)|= |\phi(t)| \ge  \pi/2$. Now by applying Lemma~\ref{vdC} to the intervals $I'_{1}$ and $I'_{2}$  we have
\[  \left|   \int_{I'_{j}} e^{ i \phi(t)}\, \d{t}  \right| \le
8 \sqrt{2/\pi}, \qquad j=1,2.  \]
Extracting the real part, we obtain 
\begin{equation}\label{Eq2}
 \left|   \int_{I'} \cos( 2\pi( xe^{t}+ y e^{-t}) ) \, \d{t}  \right| \le
16 \sqrt{2/\pi}.  
\end{equation}
On the other hand, for $t \in I$, we have $\cos(2\pi( xe^{t}+ y e^{-t}))\geq 0$ and so 
\begin{equation}\label{Eq3}
\int_{I} \cos( 2\pi( xe^{t}+ y e^{-t}) ) \, \d{t} \ge 0.
\end{equation}
Combining~\eqref{Eq2} and~\eqref{Eq3} we obtain
\[ \int_{-T}^{T} \cos( 2\pi( xe^{t}+ y e^{-t}) )\, \d{t} \ge -16 \sqrt{2/\pi}, \]
which implies that 
\[ \widehat{\mu_T}(x,y)=\frac{1}{2T} \int_{-T}^{T} \cos( 2\pi( xe^{t}+ y e^{-t}) )\, \d{t} \ge \frac{-8 \sqrt{2/\pi}}{T}. \] 
\end{proof}
%We now prove the main theorem of this subsection. 
\begin{proof}[Proof of Theorem \ref{spec-real}] Notice that $\mu_T$ is a symmetric measure with $\mathrm{Supp}(\mu_T)\subseteq \mathcal{S}_T$. Since $\mu_T$ is a probability measure then obviously we have 
\begin{equation}\label{=1}
\sup_{(x,y)\in \RR^2}\widehat{\mu_T}(x,y)=1.
\end{equation}
From Lemma~\ref{m=inf} we have $m(A_{\mu_T})=\inf_{(x,y)\in \RR^2}\widehat{\mu_T}(x,y)$, and so by Remark~\ref{positive-negative} and Lemma~\ref{ocslliatory} we have 
\begin{equation}\label{<0}
-\frac{8 \sqrt{2/\pi}}{T}\leq\inf_{(x,y)\in \RR^2}\widehat{\mu_T}(x,y)<0.
\end{equation}
Hence by Theorem~\ref{spectral}, \eqref{=1} and~\eqref{<0} we obtain
$$
1+\frac{ T}{8\sqrt{2/\pi}}\leq 1-\frac{1}{\inf_{(x,y)\in\RR^2}\widehat{\mu_T}(x,y)}\leq \bchi(\Cay(\RR^2,\mathcal{S}_T)),
$$
establishing our theorem. 
\end{proof}
For any $T>0$, obviously 
$\bchi(\RR^2,\mathcal{S}_T)\leq \bchi(\hgraph{\RR}),$
and so $\bchi(\hgraph{\RR})=\infty$ which implies that $\bchi(\qgraph{\Q})=\infty$ by Lemma~\ref{hyper-quad},  if $\Q=(V,Q)$ is an isotropic quadratic space over $\RR$.
%===========================================
\subsection{The case of isotropic $p$-adic quadratic forms }
The proof in the $p$-adic case is similar to the one in the real case. We first remark that for a $p$-adic field $F$ we have 
$\bchi(\hgraph{F})\geq \bchi(\hgraph{\QQ_p})$.
Hence in order to prove Theorem~\ref{Borel-main-theorem},
analogous to the real case,  it suffices to show that the Borel chromatic number of hyperbola graphs $\hgraph{\QQ_p}$ is infinite. 
We also recall that the ring of integers of $\QQ_p$ is denoted by $\ZZ_p$ and the maximal ideal of $\ZZ_p$ by $\p$. We will fix the uniformizer $\varpi=p$ which satisfies $\|\varpi\|_p=1/p$. Hence 
$$\p^k=\{x\in \QQ_p: \|x\|_p\leq p^{-k}\},$$ and we have the the filtration 
$$
\QQ_p\supseteq\cdots\supseteq \p^{-2}\supseteq \p^{-1}\supseteq \p^0=\ZZ_p\supseteq\p^1\supseteq \p^2\supseteq\cdots\; . 
$$
For $k\in \ZZ$, define 
$$
\C_k:=\{s\in \QQ_p: \|s\|_p=p^k\}. 
$$
Notice that $\C_k=\p^{-k}\setminus \p^{-(k-1)}$,
and so $|\C_k|=p^k-p^{k-1}$ where $|\C_k|$ denotes the Haar measure of $\C_k$.
For an integer $T\geq 1$ consider the measurable set 
\begin{equation}
\mathcal{S}_T=\left\{(x,y)\in \QQ_p^2: xy=1,\, p^{-T}\leq \|x\|_p\leq p^T\right\}.
\end{equation}
Note that $\mathcal{S}_T$ is a bounded symmetric set which does not contain the origin in its closure. Similar to the real case, the Borel chromatic number of {\it T-chopped hyperbola graph} $\Cay(\QQ_p^2,\mathcal{S}_T)$ is finite by Lemma~\ref{finite-chro}.
\begin{theorem}\label{spec-p-adic} Let $\Cay(\QQ_p^2,\mathcal{S}_T)$ be the T-chopped hyperbola graph. Then 
\begin{equation}
\bchi(\Cay(\QQ_p,\mathcal{S}_T))\geq 1+(2T+1)/4.
\end{equation}
\end{theorem}
Let $T\geq 1$ be an integer. For a measurable set $E\subseteq \QQ_p^2$ we define the following measure 
$$
\mu_T(E)=\frac{1}{L}\left(\int_{p^{-T}\leq \|s\|_p\leq p^T} \frac{\1_E(s,1/s)}{\|s\|_p}\, \d{s}+\int_{p^{-T}\leq \|s\|_p\leq p^T} \frac{\1_{(-E)}(s,1/s)}{\|s\|_p}\, \d{s}\right),
$$   
where $L=(4T+2)(1-1/p)$. It is easy to verify that  the probability measure $\mu_T$ is a symmetric measure with $\mathrm{Supp}(\mu_T)\subseteq \mathcal{S}_T$. By a straightforward calculation we  obtain 
\begin{equation}
\widehat{\mu_T}(x,y)=\frac{1}{L}\int_{p^{-T}\leq \|s\|_p\leq p^{T}}(\psi(xs+y/s)+\overline{\psi}(xs+y/s))\frac{\d{s}}{\|s\|_p},
\end{equation} 
where $\psi$ is an additive character of $\QQ_p$ with the kernel $\ZZ_p$, defined in~\eqref{additive-char}. 
 Hence our aim is to estimate $\widehat{\mu_T}(x,y)$. Following~\cite{Sally-Taibleson}, for a positive integer $r$, we define {\it incomplete Bessel function} as follows 
$$
J_1(x,y,T):=\int_{p^{-T}\leq \|s\|_p\leq p^T} \overline{\psi}(xs+y/s)\frac{\d{s}}{\|s\|_p}; \qquad\qquad x,y\in \QQ_p,\,\, T\in\mathbb{N}.
$$
 Hence 
 $$
\widehat{\mu_T}(x,y)=\frac{J_1(x,y,T)+\overline{J_1(x,y,T)}}{L}=\frac{2J_1(x,y,T)}{L},
 $$ 
since $J_1(x,y,T)$ is a real number. 
 \begin{lemma}\label{Theorem-Bessel-In}
For an integer $T\geq 1$ and arbitrary $(x,y)\in \QQ_p^2$ we have 
\begin{equation}\label{-4L-ine}
-\frac{4}{L}\leq \widehat{\mu_T}(x,y).
\end{equation}
\end{lemma}
To prove this theorem we first need a few facts involving $p$-adic integrals:
\begin{lemma} For $k\in \ZZ$ we have
\begin{equation}\label{p-adic-int-ide}
\int_{\p^k} \psi(as)\, \d{s}=
\begin{cases}
p^{-k} & \varpi^k a\in \ZZ_p,\\
0 & \text{otherwise}.
\end{cases}
\end{equation}
\end{lemma}
\begin{proof}
Let $\varpi^k a\in \ZZ_p$. Then $as\in \ZZ_p$ for all $s\in \p^k$ and so $\psi(as)=1$ which implies the first case. Suppose $\varpi^k a \not\in \ZZ_p$. Then $\psi_a(s):=\psi(as)$ is a non-trivial character on $\p^k$ and so its integral over $\p^k$ is zero.
\end{proof}
A direct consequence of this lemma is the following identity.
\begin{lemma}\label{a-identity} For $a\in \QQ_p$ with $\|a\|_p=p^{n}$ and $k\in \ZZ$ we have
\begin{equation}\label{u-ide}
\int_{\C_k} \psi(as)\frac{\d{s}}{\|s\|_p}=\int_{\C_k} \overline{\psi}(as)\frac{\d{s}}{\|s\|_p}=
\begin{cases}
1-\frac{1}{p} & k\leq -n,\\
-\frac{1}{p} & k=1-n,\\
0 & \text{otherwise}.
\end{cases}
\end{equation}
\end{lemma}
\begin{proof}
Since $\C_k=\p^{-k}\setminus \p^{-(k-1)}$ then
$$
\int_{\C_k}\psi(as)\, \d{s}=\int_{\p^{-k}}\psi(as)\, \d{s}-\int_{\p^{-(k-1)}}\psi(as)\, \d{s},
$$ 
and so by~\eqref{p-adic-int-ide} we obtain the equalities. 
\end{proof}
For a positive integer $r$ and $0\neq w\in \QQ_p$, define 
$$F(r,w)=\int_{\C_r}\psi(s)\psi(w/s)\frac{\d{s}}{\|s\|_p}=\int_{\C_r}\overline{\psi}(s)\overline{\psi}(w/s)\frac{\d{s}}{\|s\|_p}.$$
Evidently we have 
\begin{equation}\label{Evidently}
|F(r,w)|\leq 1-\frac{1}{p}\leq 1.
\end{equation}
\begin{lemma}\label{Sally} Suppose that $\|w\|_p=p^m$ and $1\leq r< m$. Then $F(r,w)\not\equiv 0$ as a function of $w$ if and only if $m$ is even and $r=m/2$.
\end{lemma}
\begin{proof}
See~\cite[Lemma 8(i)]{Sally-Taibleson} also see~\cite[Lemma 10.10]{Sally}. 
\end{proof}
\begin{proof}[Proof of Lemma~\ref{Theorem-Bessel-In}] Throughout this proof we will let $T\geq 1$ to be an integer. We show that for any $(x,y)\in\QQ_p^2$:
\begin{equation}\label{-2-in}
J_1(x,y,T)\geq -2,
\end{equation}
which from it we immediately deduce~\eqref{-4L-ine}. 
Let $x\neq 0$, then from Lemma~\ref{a-identity} we have
$$
J_1(x,0,T)=\int_{p^{-T}\leq \|s\|_p\leq p^T} \overline{\psi}(xs)\frac{\d{s}}{\|s\|_p}\, \geq -\frac{1}{p}.
$$
Now let $y\neq 0$. Since $\frac{\d{s}}{\|s\|_p}$ is a Haar measure on the multiplicative group $\QQ_p^\times$ then we have 
\begin{equation}
J_1(0,y,T)=\int_{p^{-T}\leq \|s\|_p\leq p^T} \overline{\psi}(y/s)\frac{\d{s}}{\|s\|_p}=\int_{p^{-T}\leq \|s\|_p\leq p^T} \overline{\psi}(ys)\frac{\d{s}}{\|s\|_p},
\end{equation}
and so again by Lemma~\ref{a-identity} we obtain
$$
J_1(0,y,T)\geq -\frac{1}{p}.
$$
Therefore when $xy=0$ we obtain~\eqref{-2-in}. 
In the rest of the proof, we can thus assume that $xy\neq 0$. Let $\|y\|_p=p^{n_2}$. Using a change of variables we obtain 
$$
J_1(x,y,T)=\int_{q^{-T-n_2}\leq \|s\|_p\leq p^{T-n_2}} \overline{\psi}(xys+1/s)\frac{\d{s}}{\|s\|_p}=\int_{p^{\ell_1}\leq \|s\|_p\leq p^{\ell_2}} \overline{\psi}(as+1/s)\frac{\d{s}}{\|s\|_p},
$$
where $\ell_1=-(T+n_2)$, $\ell_2=T-n_2$ and $a=xy$. First assume $\|a\|_p\leq 1$ and denote 
$$I_1:=\{s\in \QQ_p: p^{\ell_1}\leq \|s\|_p\leq p^{\ell_2},\; \|s\|_p\leq 1\},\qquad I_2:=\{s\in \QQ_p: p^{\ell_1}\leq \|s\|_p\leq p^{\ell_2},\; \|s\|_p> 1\}.$$
Hence from Lemma~\ref{a-identity} we deduce that 
\begin{equation}
\begin{split}
J_1(x,y,T)&=\int_{I_1} \overline{\psi}(as+1/s)\frac{\d{s}}{\|s\|_p}+\int_{I_2} \overline{\psi}(as+1/s)\frac{\d{s}}{\|s\|_p}\\
&=\int_{I_1} \overline{\psi}(1/s)\frac{ds}{\|s\|_p}+\int_{I_2} \overline{\psi}(as)\frac{\d{s}}{\|s\|_p}\\
&=\int_{I_1^{-1}} \overline{\psi}(s)\frac{\d{s}}{\|s\|_p}+\int_{I_2} \overline{\psi}(as)\frac{\d{s}}{\|s\|_p}\geq -2/p. 
\end{split}
\end{equation}
Therefore when $\|xy\|_p\leq 1$ we obtain~\eqref{-2-in}.
Next assume that $\|xy\|_p=\|a\|_p>1$ and define 
\begin{equation}
\begin{split}
I_1'&=\{s\in \QQ_p: p^{\ell_1}\leq \|s\|_p\leq p^{\ell_2},\,\,\, \|s\|_p\geq 1\}\\
I_2'&=\{s\in \QQ_p: p^{\ell_1}\leq \|s\|_p\leq p^{\ell_2},\,\,\, \|s\|_p\leq 1/\|a\|_p\}\\
I_3'&=\{s\in \QQ_p: p^{\ell_1}\leq \|s\|_p\leq p^{\ell_2},\,\,\, 1/\|a\|_p<\|s\|_p< 1\}
\end{split}
\end{equation} 
Hence again form Lemma~\ref{a-identity} we have 
\begin{equation}\label{eq35}
\begin{split}
J_1(x,y,T)&=\int_{I'_1} \overline{\psi}(as)\frac{\d{s}}{\|s\|_p}+\int_{I'_2} \overline{\psi}(1/s)\frac{\d{s}}{\|s\|_p}+\int_{I'_3} \overline{\psi}(as+1/s)\frac{\d{s}}{\|s\|_p}\\
&=\int_{I'_1} \overline{\psi}(as)\frac{\d{s}}{\|s\|_p}+\int_{{I'_2}^{-1}} \overline{\psi}(s)\frac{\d{s}}{\|s\|_p}+\int_{I'_3} \overline{\psi}(as+1/s)\frac{\d{s}}{\|s\|_p}\\
&\geq -2/p+\int_{I'_3} \overline{\psi}(as)\overline{\psi}(1/s)\frac{\d{s}}{\|s\|_p}.
\end{split}
\end{equation}
To complete the proof of Lemma~\ref{Theorem-Bessel-In} it remains to evaluate the last integral. We remark that 
$$\int_{I'_3} \overline{\psi}(as)\overline{\psi}(1/s)\frac{\d{s}}{\|s\|_p},$$
 is a real number. Let $\|a\|_p=p^m$ for some $m\geq 1$. Hence we obtain the following finite sum  
\begin{equation}
\begin{split}
\int_{I'_3} \overline{\psi}(as)\overline{\psi}(1/s)\frac{\d{s}}{\|s\|_p}&=\sum_{k}\int_{\C_k} \overline{\psi}(as)\overline{\psi}(1/s)\frac{\d{s}}{\|s\|_p}
=\sum_{k}\int_{\C_{k+m}} \overline{\psi}(s)\overline{\psi}(a/s)\frac{\d{s}}{\|s\|_p}.
\end{split}
\end{equation}
notice that $1\leq m+k< m$ since $s\in I_3'$. Therefore by Lemma~\ref{Sally} and~\eqref{Evidently} we have
$$
\int_{I'_3} \overline{\psi}(as)\overline{\psi}(1/s)\frac{\d{s}}{\|s\|_p}\geq -(1-1/p),
$$
which by combining with~\eqref{eq35} proves~\eqref{-2-in} when $\|xy\|_p>1$. 
\end{proof}
Since $\mu_T$ is a probability measure then
$\sup_{(x,y)\in \QQ_p^2}\widehat{\mu_T}(x,y)=1$. 
Moreover, similar to the real case and Lemma~\ref{Theorem-Bessel-In} we also have 
$$-\frac{4}{L}\leq \inf_{(x,y)\in\QQ_p^2}\widehat{\mu_T}(x,y)< 0.$$
Therefore from Theorem~\ref{spectral} we have 
$$
1+\frac{2T+1}{4}\leq 1+\frac{L}{4}\leq\bchi(\Cay(\QQ^2,\mathcal{S}_T)),
$$
which proves Theorem~\ref{spec-p-adic}. Now evidently we have 
$$
\sup_{T\geq 1}\bchi(\Cay(\QQ_p^2,\mathcal{S}_T))\leq \bchi(\hgraph{\QQ_p}),
$$
and so $\bchi(\hgraph{\QQ_p})=\infty$. 
%========================================================================
\subsection{The case of anisotropic quadratic forms}\label{anisotropic}
Let $F$ be a local field and let $\Q=(V,Q)$ be an anisotropic quadratic space with $\dim V \ge 2$. In this subsection we will show that $ \bchi(\qgraph{\Q}) < \infty$. 
\begin{proof}[Proof of Theorem \ref{Borel-main-theorem} for anisotropic forms]
Fix a norm on $F$ and denote this norm as well as the associated supremum norm on $V$ by $\|\cdot\|$. Set $S=\{v\in V: Q(v)=1\}$, hence $\qgraph{\Q}=\Cay(V,S)$. We first show that $S$ is a bounded set by proving that there exists a constant $A>0$ such that for every $v \in V$, if $\|v \| \ge A$, then $\|Q(v)\|>1$. By way of contradiction, assume that there exists a sequence 
$\{ v_m \}_{m \ge 1}$ in $V$ such that $\|v_m \| \to \infty$ as $m \to \infty$, while 
$\| Q(v_m) \|\le 1$ for all $m \ge 1$. For $m \ge 1$, choose $ \lambda_m \in F$ such that 
$\| \lambda_m\| \| v_m \|= \| \lambda_m v_m \| = 1$.
Hence $\lambda_m\to 0$, from which it follows that $Q( \lambda_m v_m)\to 0$. Let $w_m= \lambda_m v_m$. Since the unit sphere $\{ v \in V: \|v\|=1 \}$ is compact, after passing to a subsequence, we may assume that $w_m \to w$, where $\|w\|=1$. The limit point $w$ satisfies $ Q(w)=0$, which contradicts the assumption that $Q$ is anisotropic. Therefore $S$ is a bounded and measurable and so by Lemma~\ref{finite-chro}, $\bchi(\qgraph{\Q})$ is finite. 
\end{proof}

%============================================
 \section{Application to the chromatic number of regular graphs}\label{application}
Let $R$ be a ring (always assumed to contain the unity) and $\zer(R)$ denote the set of zero-divisors in $R$; set also $\reg(R)= R \setminus \zer(R)$. The {\it regular graph} of $R$, first introduced in~\cite{anderson}, is the graph with the vertex set $\reg(R)$, in which distinct vertices $x,y \in \reg(R)$ are declared to be adjacent when $x+y \in \zer(R)$. The  special case $R=\M_n(F)$, where $F$ is a field and $\M_n(F)$ is the ring of $n\times n$ matrices over $F$. Here, the vertex set coincides with the general linear group $\GL_n(F)$, with two vertices $x,y \in \GL_n(F)$ forming an edge if $\det(x+y)=0$. We will denote this graph by $\Gamma_n(F)$. 
The study of these regular graphs was initiated in \cite{AJF}, where the authors proved that the clique number of $\Gamma_n(F)$
is bounded by a function of $n$, independent of the underlying field $F$, assuming that the characteristic of $F$ is not $2$. 
More precisely, it is shown in~\cite{AJF} that for a field $F$ of characteristic different from $2$, we have
\begin{equation}\label{clique-bound}
  \omega(\Gamma_n(F) )\leq  C_{n}, 
\end{equation}
where $C_{n}= \sum_{k=0}^{n}k! {n \choose k}^2$ and $\omega(\G)$ denotes the clique number of the graph $\G$. Note that $C_{n}$ does not depend on $F$. Since $\omega(\mathcal{G}) \le \chi(\mathcal{G})$ for every graph $\G$, it would be interesting to know whether or not an inequality of the form~\eqref{clique-bound} can be established that is also independent of $F$.  We can now prove Theorem \ref{app}:

\begin{proof}[Proof of Theorem \ref{app}]
Without loss of generality we can assume that $n=2$. 
For $x,y\in F$ define
\begin{equation}\label{axy}
a_{x,y}:= 
\begin{pmatrix}
1  & -2x    \\
 0 & 1 \\
\end{pmatrix} \begin{pmatrix}
1  & 0    \\
 2y & 1 \\
\end{pmatrix} = \begin{pmatrix}
1-4xy  & -2x    \\
 2y & 1 \\
\end{pmatrix} \in \GL_2(F).
\end{equation}
A simple computation shows that 
$$\det(a_{x_1,y_1}+a_{x_2,y_2})=4-4(x_2-x_1)(y_2-y_1).$$
Hence, vertices $a_{x_1,y_1}$ and $a_{x_2,y_2}$ of $\Gamma_{2}(F)$ are adjacent if and only if $(x_2-x_1)(y_2-y_1)=1$. Note also that since $F$ has characteristic zero, the map $(x,y) \mapsto a_{x,y}$ is a injective and continuous. This implies that the induced subgraph on the set 
$A= \{ a_{x,y}: x,y \in F \},$ is isomorphic to $\hgraph{F}$. Note that the alternative description of $A$ as
\[ A= \left\{ \begin{pmatrix}
x  & y    \\
 z & 1 \\
\end{pmatrix}: x,y,z \in F, x-yz=1 \right\}, \]
makes it clear that $A$ is closed, and hence Borel. This implies that the restriction of any finite Borel coloring of 
$\Gamma_{2}(F)$ to $A$ would yield in a Borel coloring of $\hgraph{F}$ and so $$\bchi(\Gamma_2(F))\geq \bchi(\hgraph{F})=\infty,$$ by Theorem~\ref{Borel-main-theorem}.
\end{proof}
%======================================================
\subsection*{Acknowledgement} During the completion of this work, M.B. was supported by a postdoctoral fellowship from the University of Ottawa. He wishes to thank
his supervisor Vadim Kaimanovich. The authors are grateful to Saeed Akbari, Boris Bukh, Camelia Karimianpour and Hadi Salmasian with whom the authors discussed various parts of this paper. 
%========================================================
\bibliographystyle{plain}

%\bibliography{chromatic}
%======================================
\Addresses
\end{document}